\documentclass{article}

\usepackage{amsthm}
\usepackage{epsfig}

\newtheorem{thm}{Theorem}

\def\0{{\bf 0}}
\def\1{{\bf 1}}
\def\d{{\ell}}

\title{Hoffman's ratio bound}

\author{Willem H. Haemers\thanks{e-mail haemers@uvt.nl}
\\
{\it\small Department of Econometrics and Operations Research,}
\\
{\it\small Tilburg University, Tilburg, The Netherlands}
\\[10pt]
{\em In memory of Alan J. Hoffman}
}

\begin{document}
	
\date{}
\maketitle

\begin{abstract}
\noindent
Hoffman's ratio bound is an upper bound for the independence number of a regular graph in terms of the eigenvalues of the adjacency matrix.
The bound has proved to be very useful and has been applied many times.
Hoffman did not publish his result, and for a great number of users the emergence of Hoffman's bound is a black hole. 
With this note I hope to clarify the history of this bound and some of its generalizations.
\\[3pt]
\end{abstract}

\section{Introduction}

Although Hoffman didn't publish his bound, it is perhaps his most cited result.
He communicated his bound with Jaap Seidel, who then called it
the Hoffman bound (later Chris Godsil used {\lq}ratio bound{\rq} so {\lq}Hoffman's
ratio bound{\rq} seems to be a good compromise).
I was a student of Jaap Seidel in the seventies, and Jaap told me the story of the bound.
In later years, I noticed that many authors had problems with the correct reference.
I have seen references to Delsarte~\cite{Del}, Lov\'asz~\cite{Lov} and to the wrong paper of Hoffman~\cite{Hof}.
Thus, I thought that it may be a good idea to write my memories of Hoffman's bound and devote them
to his memory.

\section{The bound}

Throughout $G$ is a simple graph of order $n$, having adjacency matrix $A$ with eigenvalues
$\lambda_1\geq\cdots\geq\lambda_n$.
A subset of the vertex set of $G$ with no edges is a {\em coclique} or {\em independent set}.
The maximum size of an independent set in $G$ is the {\em independence number}, denoted by $\alpha$.

\begin{thm}\label{H}
{\rm (Hoffman)}
If $G$ is regular of degree $k$, then
\begin{eqnarray}\label{Hoffman}
\alpha \leq n\frac{-\lambda_n}{k-\lambda_n}.
\end{eqnarray}
\end{thm}

\begin{proof}
Since $G$ is regular, $A\1=k\1$ ($\1$ is the all-ones vector).
In fact, $k=\lambda_1$ by the Perron-Frobenius theorem.
So $A$ and the $n\times n$ all-ones matrix $J_n$ have a common basis of eigenvectors,
consisting of $\1$ and $n-1$ vectors orthogonal to $\1$.
This implies that the smallest eigenvalue of $A - \frac{k-\lambda_n}{n}J_n$ equals $\lambda_n$,
and hence $E=A-\frac{k-\lambda_n}{n}J_n - \lambda_nI_n$ is positive semi-definite.
Since $G$ has a coclique of order $\alpha$,
the $\alpha\times\alpha$ matrix $E_\alpha = -\frac{k-\lambda_n}{n}J _\alpha - \lambda_n I_\alpha$ is a principal submatrix of $E$.
Therefore, $E_\alpha$ is positive semi-definite, which implies $-\lambda_n-\alpha(k-\lambda_n)/n \geq 0$,
hence $\alpha\leq -\lambda_n n/(k-\lambda_n)$.
\end{proof}

\section{History}
The story starts with Philippe Delsarte.
In 1973 he published his influential PhD thesis~\cite{Del}
entitled {\lq}An algebraic approach to the association schemes of coding theory{\rq}.
An important result in \cite{Del} is the so-called {\em linear programming bound} for subsets in an association scheme.
As an example of this linear programming bound Delsarte derives the following upper bound for $\omega$,
the maximum order of a clique in a strongly regular graph $G$ (formula~(3.22) in \cite{Del}):
\begin{eqnarray}\label{Delsarte}
\omega \leq 1-k/\lambda_n.
\end{eqnarray}
A clique in $G$ is a coclique in the complement of $G$, and for a strongly regular graph the complement
is also strongly regular and has degree $n-k-1$ and smallest eigenvalue $-\lambda_2-1$.
Therefore $\alpha \leq 1+(n-k-1)/(\lambda_2+1)$.
The eigenvalues of a strongly regular graph satisfy $(k-\lambda_2)(k-\lambda_n)=(k+\lambda_2\lambda_n)n$.
By use of this identity 
it follows that the latter inequality is Hoffman's bound (\ref{Hoffman}).
Thus, if applied to a strongly regular graph, Delsarte~\cite{Del} is the right reference for~(\ref{Hoffman}).

The above derivation can be a bit cumbersome,
and it may come as a surprise that the final formula is as simple as Hoffman's bound.
But in fact, there is a reason.
A strongly regular graph is a special case of a graph in an association scheme and Delsarte has proved that
for such  graph the product of the linear programming bounds for a clique and a coclique is at most $n$.
Therefore $\alpha\leq n/(1-k/\lambda_n)$.
\\

Jaap Seidel was a good friend and close collaborator of Jean-Marie Goethals, who was the supervisor of Delsarte.
So Jaap knew Delsarte and his work very well and he liked the bounds for a clique and coclique in a strongly regular graph.
At some conference in which Jaap Seidel and Alan Hoffman both participated
(I don't remember if Jaap mentioned to me which conference it was;
perhaps it was the NATO meeting on Combinatorics which took place at Nijenrode castle in Breukelen, The Netherlands, in 1974),
Jaap and Alan were sitting together talking mathematics.
Jaap told Alan about Delsarte's result.
Then, a little later Alan replied that he could prove that the coclique bound is valid for arbitrary regular graphs,
and moreover, that it would give a short proof for his lower bound on the chromatic number $\chi$ ($\chi\geq 1-\lambda_1/\lambda_n$;
see \cite{Hof}) in case the graph is regular.
Indeed, in a vertex coloring each color class has maximum size at most $\alpha$, so the number of color classes
is at least $n/\alpha \geq 1-k/\lambda_n$.
In 1975, Jaap told me the story and he wondered if the Hoffman bound could be generalized to arbitrary graphs.
It was the starting point of my PhD research.
\\

In 1979 L\'aszl\'o Lav\'asz published his famous paper {\lq}On the Shannon capacity of a graph{\rq}~\cite{Lov}.
In that paper Lov\'asz introduced the graph parameter $\vartheta$ (now known as {\lq}Lov\'asz's theta{\rq})
and proved that $\vartheta$ is an upper bound for the {\em Shannon capacity} $\Theta$, a measure from information theory.
In~\cite[Theorem~9]{Lov} Lov\'asz states that $\vartheta\leq -n\lambda_n/(k-\lambda_n)$.
By definition $\Theta$ is bounded from below by the independence number $\alpha$.
Therefore Lov\'asz' Theorem~9 implies Hoffman's bound.
Lov\'asz was aware of this and mentioned it in the acknowledgements.
However, he gives Hoffman too much credit and attributes Theorem~9 to him.

It is interesting that Hoffman was very close to proving that his bound is also an upper bound
for the Shannon capacity.
It follows easily from the proof of Theorem~\ref{Hoffman} that the Hoffman bound remains valid for a
weighted adjacency matrix $A$, as long as all row (and column) sums are equal to $k$.
Using this observation it takes just a few steps to prove that $\Theta\leq -\lambda_n n/(k-\lambda_n)$, as is illustrated in Theorem~3.4 of \cite{Hae3}.

Because $\vartheta$ is defined for an arbitrary graph, $\vartheta$ can be seen as a generalization of Hoffman's bound
to arbitrary graphs.
However, the definition of $\vartheta$ is involved
and what Jaap wanted was an easy formula in terms of the eigenvalues of the adjacency matrix.

\section{Interlacing eigenvalues}
The generalization that Jaap had in mind appeared in my PhD thesis~\cite{Hae2},
where I used eigenvalue interlacing to obtain a number of Hoffman-type inequalities.
An announcement of the results appeared in \cite{Hae1} in 1978, and this paper is perhaps the
earliest paper that mentions Hoffman's bound.
Below is a simplified illustration of this approach.
\begin{thm}\label{interlacing} {\rm\cite{Hae2,Hae3}}
Consider the matrices
\[
A=\left[
\begin{array}{cc}
A_{1,1} & A_{1,2} \\ A_{2,1} & A_{2,2}
\end{array}
\right]
\mbox{ and }
B=
\left[
\begin{array}{cc}
b_{1,1} & b_{1,2} \\ b_{2,1} & b_{2,2}
\end{array}
\right],
\]
where $A$ is a real symmetric matrix of order $n$ partitioned such that $A_{1,2}=A_{2,1}^\top$,
and $b_{i,j}$ is the average row sum of $A_{i,j}$. 
Let $\lambda_1\geq\cdots\geq\lambda_n$ be the eigenvalues of $A$, and let $\mu_1\geq \mu_2$ be the eigenvalues of $B$.
Then
\\[3pt]
(i)\ \ (interlacing) $\lambda_1\geq\mu_1\geq\lambda_{n-1}$, \ and \ $\lambda_2\geq\mu_2\geq\lambda_n$,
\\[3pt]
(ii) if $(\mu_1,\mu_2) = (\lambda_1,\lambda_2)$, or $(\lambda_1,\lambda_n)$, or $(\lambda_{n-1},\lambda_n)$ then
each block $A_{i,j}$ has constant row and column sum.
\end{thm}
We apply Theorem~\ref{interlacing} to the adjacency matrix $A$ of $G$.
Suppose $A_{1,1}$ corresponds to a coclique $C$ in $G$ of order $\alpha$.
Then $A_{1,1}=O$, $b_{1,1}=0$, $\alpha\, b_{1,2} = b_{2,1}(n-\alpha)$ and therefore
$\det B  = -\alpha\, b_{1,2}^2/(n-\alpha)$.
On the other hand, Theorem~\ref{interlacing}(i)  gives
\[
 -\det B = -\mu_1\mu_2 \leq -\lambda_1\lambda_n.
\]
Hence $\alpha \leq -n\lambda_1\lambda_n/(b_{1,2}^2 - \lambda_1\lambda_n)$.
If $\delta$ is the minimum degree in $G$ then $b_{1,2} \geq \delta$, and therefore:
\begin{eqnarray}\label{H1}
\alpha\leq n\frac{-\lambda_1\lambda_n}{\delta^2-\lambda_1\lambda_n}.
\end{eqnarray}
If $G$ is regular of degree $k$, then $k=\lambda_1=\delta$, so in this case
(\ref{H1}) reduces to Hoffman's ratio bound.

If equality holds in (\ref{H1}), then $\mu_1=\lambda_1$ and $\mu_2=\lambda_n$, so Theorem~\ref{interlacing}(ii)
gives that each vertex of $G$ has a constant number of neighbors inside $C$, and a constant number of neighbors outside $C$.
Clearly these constants are the entries of $B$.
In particular, in case $G$ is regular we find that equality in the Hoffman bound implies
that each vertex outside $C$ has exactly $\alpha k/(n-\alpha) = -\lambda_n$ neighbors in $C$.

If $A_{1,1}$ corresponds to an arbitrary induced subgraph of $G$ something similar can be done,
and if $G$ is regular of degree $k$, the formula becomes rather neat.
Then $k=\lambda_1=\mu_1$, and Theorem~\ref{interlacing} gives $\lambda_2\geq\mu_2\geq\lambda_n$
which leads to
\[
\lambda_2\geq \frac{n\d - mk}{n-m} \geq \lambda_n,
\]
where $m$ is the order, and $\d$ is the average degree of the induced subgraph corresponding to $A_{1,1}$.
If $\d=0$, the right hand side gives Hoffman's bound, and if $\d=m-1$ the left hand side gives
$m \leq n(\lambda_2+1)/(\lambda_2+n-k)$, which is Hoffman's bound applied to the complement of $G$.

Much later Godsil and Newman~\cite[Corollary 3.6]{GN}
generalized Hoffman's bound to arbitrary graphs in terms of the eigenvalues of the {\em Laplacian matrix} $L$ defined by $L=D-A$,
where $D$ is the diagonal matrix with the vertex degrees.
Also for this case interlacing works.
We apply Theorem~\ref{interlacing} with $L$ instead of $A$, such that $L_{1,1}$ corresponds to a coclique $C$ of order $\alpha$.
Then $L_{1,1}$ is a diagonal matrix with the degrees of the vertices in $C$ on the diagonal.
This implies that $b_{1,1}$, the average row sum of $L_{1,1}$, is at least the minimum degree $\delta$.
From $L\1=\0$, it follows that $B\1=\0$, which implies that $\mu_2=0$, $\mu_1={\rm trace}~B = b_{1,1}+b_{2,2}$
and $b_{2,2}=\alpha\, b_{1,1}/(n-\alpha)$.
Let $\varphi$ be the largest eigenvalue of $L$, then Theorem~\ref{interlacing}(i) gives $\mu_1\leq\varphi$,
and therefore:
\[
\alpha\leq n(1-\delta/\varphi).
\]
If equality holds, we get the same conclusion as for equality in (\ref{H1}).
If $G$ is regular of degree $k$, then $\delta=k$ and $\varphi=k-\lambda_n$, and again we obtain Hoffman's ratio bound.

\end{document}